\documentclass[11pt,oneside]{amsart}
\usepackage{amsmath,amsthm,amssymb,amscd,fancybox}
\usepackage[a4paper,left=3cm,right=3cm,top=3cm,bottom=3cm]{geometry}
\usepackage{cite}
\renewcommand{\epsilon}{\varepsilon}
\renewcommand{\phi}{\varphi}

 \newcommand{\Lp}{L_{\vec{p}}(\bR^n)}

 \newcommand{\bZ}{\mathbb{Z}}
\newcommand{\bR}{\mathbb{R}} \newcommand{\bN}{\mathbb{N}}
\newcommand{\cF}{\mathcal{F}} 
\newcommand{\dS}{\mathcal{S}(\bR^n)} 

\newcommand{\supp}{\operatorname{supp}}
\newcommand{\qp}{\mathcal{Q}}
\newcommand{\Zd}{\mathbb{Z}^n}
\newcommand{\Zn}{\mathbb{Z}^n}
\newcommand{\Rn}{\mathbb{R}^n}

\newcommand{\brac}[1]{\langle #1\rangle}
\newtheorem{theorem}{Theorem}[section]
\newtheorem{lemma}[theorem]{Lemma}
\newtheorem{proposition}[theorem]{Proposition}

\newtheorem{corollary}[theorem]{Corollary}
\theoremstyle{definition}
\newtheorem{definition}[theorem]{Definition}
\newtheorem{example}[theorem]{Example}
\theoremstyle{remark}
\newtheorem{remark}[theorem]{Remark}
\numberwithin{equation}{section}
\renewcommand{\d}{\, \text{d}}

\def\bZ{{\mathbb Z}}

\def\bN{{\mathbb N}}

\def\bC{{\mathbb C}}
\def\bR{{\mathbb R}}

\def\cF{\mathcal{F}}

\def\cM{\mathcal{M}}

\begin{document}
\title[Discrete Decomposition of Mixed-Norm $\alpha$-modulation spaces]{ Discrete Decomposition of Mixed-Norm $\alpha$-modulation spaces}
\author[M.\ Nielsen]{Morten Nielsen}
\address{Department of Mathematical Sciences\\ Aalborg
  University\\ Skjernvej 4A\\ DK-9220 Aalborg East\\ Denmark}
\email{mnielsen@math.aau.dk}
\begin{abstract}
    We study a class of almost diagonal matrices compatible with the mixed-norm $\alpha$-modulation spaces $M_{\vec{p},q}^{s,\alpha}(\mathbb{R}^n)$, $\alpha\in [0,1]$, introduced recently by Cleanthous and Georgiadis [Trans.\ Amer.\ Math.\ Soc.\ 373 (2020), no. 5, 3323-3356].  The class of almost diagonal matrices is shown to be closed under matrix multiplication and we connect the theory to the continuous case by identifying a suitable notion of  molecules for the mixed-norm $\alpha$-modulation spaces. We show that the  "change of frame" matrix for a pair of  time-frequency frames for the mixed-norm $\alpha$-modulation consisting of suitable molecules is almost diagonal. As examples of applications, we  use the almost diagonal matrices to construct compactly supported frames for the mixed-norm $\alpha$-modulation spaces, and to obtain a straightforward boundedness result for Fourier multipliers on the mixed-norm $\alpha$-modulation spaces. 

\end{abstract}
\maketitle

\section{Introduction}

The $\varphi$-transform approach introduced by Frazier and Jawerth, see \cite{Frazier1985,Frazier1990}, for the study of Besov and Triebel-Lizorkin spaces on $\bR^n$ provides a powerful framework for analysing smoothness spaces using discrete decompositions, and their general approach has since been extended and adapted to cover smoothness spaces in various other setting, see, e.g., \cite{MR3981281,Bownik2006,Nielsen2012}. One advantage of the discrete approach is that it opens for a more straightforward way to study boundedness of various linear operators, such as Fourier multipliers and pseudo-differential operators, by studying their matrix representation relative to the $\varphi$-transform decomposition system. An essential component to facilitate the matrix approach is to have an algebra of almost diagonal matrices available. 

The desire to  better analyse anisotropic phenomena using models based on e.g.\ partial-differential or pseudo-differential operators has recently generated  considerable interest function spaces in anisotropic and mixed-norm settings, see for example \cite{AI, BB, MR2423282, Bownik2006,Bownik2005, MR3622656,MR3573690, JS,Cleanthous2019,MR3573690,MR2319603} and reference therein. It is therefore of interest to adapt the  $\varphi$-transform machinery to the various 
anisotropic and mixed-norm settings that has proven to be useful for applications, and the purpose of the present paper is to extend the $\varphi$-tranform machinery to the family of mixed-norm $\alpha$-modulation spaces $M_{\vec{p},q}^{s,\alpha}(\mathbb{R}^n)$, $\alpha\in [0,1]$, introduced recently by Cleanthous and Georgiadis \cite{MR4082240}.

The $\alpha$-modulation spaces is a family of smoothness spaces that contains the Besov spaces, and the
modulation spaces introduced by Feichtinger \cite{fei}, as special
cases. The family offers the flexibility to tune general time-frequency properties measured by the smoothness norm by adjusting the $\alpha$-parameter.  
In the non-mixed-norm setting, the $\alpha$-modulation spaces were introduced by Gr\"obner
\cite{Groebner1992}.   Gr\"obner used the general framework of decomposition type
Banach spaces considered by Feichtinger and Gr\"obner in
\cite{MR87b:46020,MR89a:46053} to build the $\alpha$-modulation
spaces.  The parameter $\alpha$ determines a specific type of
decomposition of the frequency space $\bR^n$ used to define
the $\alpha$-modulation $M^{s,\alpha}_{p,q}(\bR^n)$.

 A $\varphi$-transform was introduced and studied for the mixed-norm $\alpha$-modulation spaces by Cleanthous and Georgiadis \cite{MR4082240}. In the present paper, we identify a suitable algebra of almost diagonal matrices compatible with the $\varphi$-transform for the mixed-norm $\alpha$-modulation spaces $M_{\vec{p},q}^{s,\alpha}(\mathbb{R}^n)$. The matrices are  defined using a joint time-frequency localization condition that was first considered for general decomposition spaces by Rasmussen and the author in \cite{Nielsen2012}. The definition of the almost diagonal matrices also leads to a natural definition of  time-frequency molecules for the spaces, and it is shown that the "change-of-frame"-matrix between two suitable families of molecules will automatically be almost diagonal.  
 
 The algebra of almost diagonal matrix approach can be used as a tool to study a great variety of topics. To support this claim, we consider two specific applications that should be of independent interest. Firstly, we use the almost diagonal matrices to obtain a stable perturbation result for the frame formed by the $\varphi$-system. We use the perturbation result to prove that there exist compactly supported frames for the mixed-norm $\alpha$-modulation spaces. Compactly supported frames are known to be useful for e.g.\ numerical purposes. Secondly, we use the molecular approach to obtain a "painless" boundedness result for Fourier multiplies on the mixed-norm $\alpha$-modulation spaces, where "painless" refers to the fact that only a few very elementary integral estimates are needed for the result. We mention that a more technically involved approach to Fourier multiplies is considered in   \cite{MR4082240}.
 
 Another potentially useful application is to pseudo-differential operators on mixed-norm $\alpha$-modulation spaces with the $\varphi$-transform approach, but we leave this study to the reader for the sake of brevity. Pseudo-differential operators on mixed-norm $\alpha$-modulation is considered by the author in \cite{}, where  another approach is used.


The structure of the paper is as follows. In Section
\ref{sec:mn} we introduce mixed-norm Lebesgue and $\alpha$-modulation spaces 
based on a so-called bounded admissible
partition of unity (BAPU) adapted to the mixed-norm setting.  Section
\ref{sec:mn} also introduces the maximal function estimates that will be needed to prove the main result. 
Section \ref{sec:ms} reviews the $\varphi$-transform introduced in  \cite{MR4082240}, and in Section \ref{sec:alm} we define the notion of molecules for mixed-norm $\alpha$-modulation spaces, along with a compatible notion of almost diagonal matrices. In Section \ref{sec:alge}, we show that the almost diagonal matrices form a suitable matrix algebra, and the boundedness of the matrices on discrete mixed-norm $\alpha$-modulation spaces is studied.  The paper concludes with Section \ref{sec:appl}, where two applications of the theory are considered. Specifically, it is shown that one can obtain compactly supported frames for   mixed-norm $\alpha$-modulation spaces using a perturbation approach, and boundednesss results for Fourier multipliers on the  mixed-norm $\alpha$-modulation spaces are considered.

\section{Mixed-norm  Spaces}\label{sec:mn}
In this section we introduce the mixed-norm Lebesgue spaces along with some needed maximal function estimates. Then we introduce mixed-norn $\alpha$-modulation spaces 
based on  so-called bounded admissible
partition of unity adapted to the mixed-norm setting. 
\subsection{Mixed-norm Lebesgue Spaces}
Let $\vec{p}=(p_1,\dots,p_n)\in(0,\infty]^n$ and $f:\mathbb{R}^n\rightarrow \mathbb{C}$. We say that a measurable function $f:\bR^n\rightarrow \bC$ belongs to $L_{\vec{p}}=L_{\vec{p}}(\mathbb{R}^n)$ provided
\begin{equation}\label{Lp}
\|f\|_{\vec{p}}:=\left(\int_\mathbb{R}\cdots\left(\int_\mathbb{R}\left(\int_\mathbb{R} |f(x_1,\dots,x_n)|^{p_1} dx_1\right)^{\frac{p_2}{p_1}} dx_2\right)^{\frac{p_3}{p_2}}\cdots dx_n\right)^{\frac{1}{p_n}}<\infty,
\end{equation}
with the standard modification when $p_j=\infty$ for some $j\in \{1,\ldots,n\}$.
Note that when $\vec{p}=(p,\dots,p),$ then $L_{\vec{p}}$ coincides with the standard $L_p$.
 One can verify that $\|\cdot\|_{\vec{p}}$ is a quasi-norm, and in fact a norm when $\min\{p_1,\dots,p_n\}\geq 1$, and turns $(L_{\vec{p}},\|\cdot\|_{\vec{p}})$ into a (quasi-)Banach space.  The quasi-norm $\|\cdot\|_{\vec{p}}$ obeys the following sub-additivity property 
 \begin{equation}\label{suba}
     \|f+g\|_{\vec{p}}^r\leq \|f\|_{\vec{p}}^r+\|g\|_{\vec{p}}^r, \qquad f,g\in L_{\vec{p}}(\mathbb{R}^n),
 \end{equation} 
 for every $0<r\leq\min\{1,p_1,p_2,\ldots,p_n\}$.
 For additional properties of $L_{\vec{p}}$, see, e.g., \cite{AI,MR370171,MR126155,Cleanthous2019}.

\subsection{Maximal operators}\label{subsec:maxop}
The maximal operator will be central to most of the estimates considered in this paper. Let $1\leq k\leq n$. We define
\begin{equation}\label{MK}
M_k f(x)=\sup\limits_{I\in I_x^k} \dfrac{1}{|I|} \int_I |f(x_1,\dots,y_k,\dots,x_n)| dy_k,\;\;f\in L^{1}_{loc}(\bR^n),
\end{equation}
where $I_x^k$ is the set of all intervals $I$ in $\mathbb{R}_{x_k}$ containing $x_k$.\

We will use extensively the following iterated maximal function:
\begin{equation}\label{Max1}
\cM_\theta f(x):=\left(M_n(\cdots(M_1|f|^\theta)\cdots)\right)^{1/\theta}(x),\;\theta>0,\;x\in\mathbb{R}^n.
\end{equation}

\begin{remark} If $Q$ is a rectangle $Q=I_1\times\dots\times I_n,$ it follows easily that for every locally integrable $f$
\begin{equation}\label{rectangle}
\int_Q |f(y)| dy\leq |Q| \cM_1 f(x)=|Q|\cM_\theta^\theta|f|^{1/\theta}(x), \ \theta>0,\ x\in\bR^n.
\end{equation}
\end{remark}

We shall need the following mixed-norm version of the maximal inequality, see \cite{MR370171,JS}:\
If $\vec{p}=(p_1,\dots,p_n)\in(0,\infty)^n$ and $0<\theta<\min\{p_1,\dots,p_n\}$ then there exists a constant $C$ such that
\begin{equation}\label{max}
\|\cM_\theta f\|_{\Lp}\leq C \|f\|_{\Lp}
\end{equation}

An important related estimate is a Peetre maximal function estimate, which will be one of our main tools in the sequel. 
For $ a = (a_1,\ldots,a_n) \in\bR^n$, $r_0>0$,  we consider  the corresponding cube $R[{a},{r_0}]$ defined as
\begin{equation}\label{eq:cube}
R[{a},{r_0}]:= [a_1-r_0,a_1+r_0]\times\cdots\times[a_n-r_0,a_n+r_0].
\end{equation}
For  $f\in L_1(\bR^n)$, we let  $$\mathcal{F}(f)(\xi):=(2\pi)^{-n/2}\int_{\bR^n}
f(x)e^{- i x\cdot\xi}\,dx,\qquad \xi\in\bR^n,$$
denote the Fourier transform, and we use the standard notation $\hat{f}(\xi)=\mathcal{F}(f)(\xi)$. With this normalisation, the Fourier transform extends to a unitary transform on $L^2(\bR^n)$ and we denote the inverse Fourier transform by $\mathcal{F}^{-1}$.

For every $\theta>0,\;$there exists a constant $c=c_\theta>0,$ such that for every $R>0$ and $f$ with $\supp(\hat{f})\subset c_f+R[-2,2]^n$, see \cite{MR3573690},
\begin{equation}\label{M3}
\sup_{y\in\mathbb{R}^n} \dfrac{|f(y)|}{\brac{R(x-y)}^{n/\theta}}\leq c\cM_\theta f(x),\;x\in\mathbb{R}^n,
\end{equation}
where the bracket is given  by  $\brac{x}:=(1+|x|^2)^{1/2}$, $x\in\bR^n$, with $|\cdot|$ denoting the Euclidean norm.
 In particular, the constant $c$ in Eq.\ \eqref{M3} is independent of the point $c_f\in \bR^n$. For a proof of the estimate, see \cite{MR3573690}.
\subsection{Mixed-norm Modulation spaces}\label{sec:mod_space}
In this section we recall the definition of mixed-norm $\alpha$-modulation spaces 
as introduced by Cleanthous and Georgiadis \cite{MR4082240}. The
$\alpha$-modulation spaces form a family of smoothness spaces that contain  modulation and Besov spaces as special ``extremal'' cases. The spaces
are defined by a parameter $\alpha$, belonging  to the interval
$[0,1]$. This parameter determines a segmentation of the frequency
domain from which the spaces are built.

\begin{definition}\label{def:cov}
A countable collection $\qp$ of measurable subsets $Q\subset \bR^n$ is called an
admissible covering of $\bR^n$ if 
\begin{itemize}
\item[i.]$\bR^n=\cup_{Q\in \qp} Q$ 
\item [ii.] There
exists $n_0<\infty$ such that 
  $\#\{Q'\in\qp:Q\cap Q'\not=\emptyset\}\leq n_0$ for all $Q\in\qp$.
\end{itemize}  
    An
  admissible covering is called an $\alpha$-covering, $0\leq
  \alpha\leq 1$, of $\bR^n$ if
 \begin{itemize}
\item[iii.] 
  $|Q|\asymp \langle x\rangle^{\alpha n}$ (uniformly) for all $x\in Q$ and for all $Q\in\qp$,
\item [iv.] There exists a constant $ K <\infty$ such that 
$$\sup_{Q\in\qp} \frac{R_Q}{r_Q}\leq K,$$
where $r_Q :=\sup\{r\in [0,\infty):\exists c_r\in\bR^n: B(c_r,r)\subseteq Q\}$
and $R_Q :=\inf\{r\in (0,\infty):\exists c_r\in\bR^n: B(c_r,r)\supseteq Q\}$, where $B(x,r)$ denotes the Euclidean ball in $\bR^n$ centered at $x$ with radius $r$.
\end{itemize}
\end{definition}
\begin{remark}
For $Q\in\qp$, condition iv.\ in Definition \ref{def:cov} ensures that we have the following  containment in cubes, where we use the notation introduced in Eq.\ \eqref{eq:cube},
$$R[\xi_1,r_Q/(2\sqrt{n})]\subseteq Q\subseteq R[\xi_2,R_Q],$$
for some $\xi_1,\xi_2\in Q$. This in turn implies that
$$|Q|\asymp R_Q^n\quad\text{ and }\quad \|\mathbf{1}_Q\|_{\vec{r}}\asymp R_Q^{\frac{1}{r_1}+ \frac{1}{r_2}+\cdots+\frac{1}{r_n}},\qquad \text{for }\vec{r}=(r_1,r_2,\ldots,r_n)\in(0,\infty]^n,$$
uniformly in $Q\in\qp$, with $\mathbf{1}_Q$ the characteristic function of $Q$.
\end{remark}

The following example was first considered in \cite{Groebner1992}.
\begin{example}\label{ex:cov}
For $\alpha\in [0,1)$, there exists $c_0>0$ such that for any $c_1\geq c_0$, the family of sets
$$B_k^\alpha:=B\big(k r_k, c_1 r_k\big),\qquad k\in\Zn,$$
with $B(c,r)$ denoting the (open) Euclidean ball of radius $r>0$ centered at $c\in\bR^n$, and $r_k:=\langle k\rangle^{\frac{\alpha}{1-\alpha}}$,  form an  $\alpha$-covering.
\end{example}

The  constants  $r_k$ together with derived constants $\xi_k$ and $x_{k,\ell}$ will often appear in  subsequent estimates, so let us fix the notation. We put
\begin{equation}
\label{eq:rk}
r_k:=\langle k\rangle^{\frac{\alpha}{1-\alpha}},\quad
\xi_k:=k r_k=k\langle k\rangle^{\frac{\alpha}{1-\alpha}},\quad
\text{ and } \quad x_{k,\ell}:=\frac{\pi}{a}r_k^{-1}\ell,\qquad r,\ell\in\Zn,
\end{equation}
with $a$ a constant that for technical reasons is often chosen such that $a\geq \max\{2c_1,\pi\sqrt{n}/2\}$,  with $c_1$ the constant from Example \ref{ex:cov}.

We will need a mixed-norm bounded admissible partition of unity adapted to $\alpha$-coverings.

\begin{definition}
Let  $\qp$ be an $\alpha$-covering of $\bR^n$, and let $\vec{p}\in (0,\infty]^n$. A corresponding
mixed-norm bounded admissible partition of unity ($\vec{p}$-BAPU) $\{\psi_Q\}_{Q\in\qp}$ is a family
of functions satisfying
\begin{itemize}
\item $\text{supp}(\psi_Q)\subset Q$
\item $\sum_{Q\in\qp} \psi_Q(\xi)=1$
\item $\sup_{Q\in \qp}|Q|^{-1}\|\,\mathbf{1}_Q\|_{L_{\vec{\tilde{p}}}}\|\mathcal{F}^{-1}\psi_Q\|_{L_{\vec{\tilde{p}}}}<\infty$,
\end{itemize}
where $\tilde{p}_j:=\min\{1,p_1,\ldots,p_j\}$ for $j=1,2,\ldots, n$, $\vec{\tilde{p}}:=(\tilde{p}_1,\ldots,\tilde{p}_n)$, and $\mathbf{1}_Q$ denotes the indicator function of the set $Q$.
\end{definition}

The results in Section
\ref{sec:ms} and the construction of a $\varphi$-transform rely on the known fact that it is possible to construct
a smooth $\vec{p}$-BAPU with certain ``nice'' properties. We summarise the needed properties in the following proposition proved in  \cite{MR4082240}, see also \cite{Borup2006a}. We let $\dS$ denote the Schwartz class of functions defined on $\bR^n$.
\begin{proposition}\label{prop:psideriv}
  For $\alpha\in [0,1)$, there exists an $\alpha$-covering of $\bR^n$
  with a corresponding $\vec{p}$-BAPU $\{\psi_k\}_{k\in\Zd}\subset \dS$ satisfying:
  \begin{itemize}
  \item[i.] $\xi_k\in Q_k$, $k\in\Zd$, with $\xi_k$  given in Eq.\ \eqref{eq:rk}.
  \item[ii.] The following estimate holds, 
  $$|\partial^\beta_\xi \psi_k(\xi)|\leq C_\beta\langle \xi\rangle^{-|\beta|\alpha},\qquad \xi\in\bR^n,$$
  for every multi-index $\beta$ with a constant $C_\beta>0$ independent of $k\in \Zd$.
  \item [iii.]  Define $\widetilde{\psi}_k(\xi) = \psi_k(|\xi_k|^\alpha\xi+\xi_k)$. Then for
every $\beta\in \bN_0^n$ there exists a constant
  $C_\beta>0$, independent of $k\in \Zd$, such that 
$$|\partial^\beta_\xi \widetilde{\psi}_k(\xi)| \leq C_\beta \mathbf{1}_{B(0,r)}(\xi).$$
\item [iv.] Define
  ${\mu}_k(\xi) = \psi_k(a_k\xi)$, where
  $a_k:=\brac{\xi_k}$. Then for every $m\in \bN$ there exists a constant
  $C_m>0$ independent of $k$ such that 
$$|\hat{{\mu}}_k(y)| \leq C_m a_k^{(m-n)(1-\alpha)}\brac{y}^{-m},\qquad y\in\bR^n.$$
  \end{itemize}
\end{proposition}

\begin{remark}
A closer inspection of the construction presented in \cite{MR4082240} reveals that the BAPU is in fact $\vec{p}$-independent and only depends on $\alpha$ through the geometry of the $\alpha$-covering.  
\end{remark}

The case $\alpha=1$ corresponds to a dyadic-covering and it is not included in the particular  construction of Proposition \ref{prop:psideriv}, but it is known that $\vec{p}$-BAPU can easily be constructed for dyadic coverings, see  e.g.\ \cite{MR3682609}.  Using $\vec{p}$-BAPUs, it is now possible to introduce the family of mixed-norm $\alpha$-modulation spaces. 

It is instructive to mention one specific simple construction of an $\vec{p}$-BAPU in order to be able to construct compatible tight frames, which will be used as for discrete decompositions.  Let $\varphi\in C^\infty(\bR^n)$ be non-negative with $\varphi(\xi)=1$ when $|\xi|\leq 1$ and $\varphi(\xi)=0$ for $|\xi|> \frac{3}{2}$, and put
$$\phi_k(\xi):=\varphi\left(\frac{\xi-r_k k}{c_1 r_k}\right),\qquad k\in\Zn,$$
where $c_1$ is the constant from Example \ref{ex:cov} and
$r_k$  given in Eq.\ \eqref{eq:rk}.
 We notice that $\phi_k(\xi)=1$ on the sets  $B_k^\alpha$ from Example \ref{ex:cov} forming an $\alpha$-cover, and, moreover, it can be verified that $\{\phi_k\}$ forms a $\vec{p}$-BAPU.
A ''square-root'' of this  $\vec{p}$-BAPU can be obtained by setting
\begin{equation}\label{eq:theta}
\theta_k^\alpha(\xi)=\frac{\phi_k(\xi)}{\sqrt{\sum_{\ell\in\Zn} \phi_\ell^2(\xi)}}    
\end{equation}
We then have
 $$\sum_{\ell\in\Zn} [\theta_k^\alpha(\xi)]^2=1,\qquad \xi\in\Rn.$$
With the notion of a $\vec{p}$-BAPU, we can now define the mixed-norm $\alpha$-modulation spaces.
\begin{definition}\label{def:Malpha}
Let $\alpha\in [0,1], s\in \bR, \vec{p}\in (0,\infty)^n, q\in (0,\infty],$ and 
let $\qp$ be an $\alpha$-covering with associated $\vec{p}$-BAPU 
   $\{\psi_k\}_{k\in\bZ^n\backslash \{0\}}$ of the type considered in Proposition \ref{prop:psideriv}. Then we define the mixed-norm $\alpha$-modulation space, $M^{s,\alpha}_{\vec{p},q}(\bR^n)$ as the set of
  tempered distributions $f\in S^\prime(\bR^n)$ satisfying
  \begin{equation}
    \label{eq:mod}
  \|f\|_{M^{s,\alpha}_{\vec{p},q}} := \biggl( \sum_{k\in\bZ^n}
r_k^{qs}\, \bigl\| \cF^{-1} (\psi_k\cF f)\bigr\|_{\Lp}^q
\biggr)^{1/q}<\infty,  
  \end{equation}
 with $r_k$  given in Eq.\ \eqref{eq:rk}.
For $q=\infty$, we
change of the sum to $\sup_{k\in\bZ^n}$.
\end{definition}

It is proved in \cite{MR4082240} that  the definition of 
$M^{s,\alpha}_{\vec{p},q}(\bR^n)$ is independent of the $\alpha$-covering and
of the particular BAPU. This important fact will allow us to build a matrix-algebra relative to the one specific $\alpha$-covering given in Example \ref{ex:cov}, and then extend to other coverings and BAPUs.  See also \cite{MR87b:46020} for a discussion of this matter in the case of general decomposition spaces. 
\section{Tight frames and the $\varphi$-transform} \label{sec:ms}

In this section we define discrete mixed-norm $\alpha$-modulation space together with a simple construction of adapted tight frames that will support a $\phi$-transform in the spirit of the classical construction by Frazier and Jawerth \cite{Frazier1985,Frazier1990}.

Let $k,\ell\in \Zn$, and let   $r_k$ be as in Eq.\ \eqref{eq:rk}. We define for a constant $a>\pi\sqrt{n}/2$, the sets
\begin{equation}\label{eq:qk}
Q(k,\ell):=
B\bigg(-\frac{\pi}{a} r_k^{-1}\ell, r_k^{-1}\bigg).
\end{equation}
It is easy to verify that there exists $0<L<\infty$ such that
$$\sum_{\ell\in\Zn} \mathbf{1}_{Q(k,\ell)}(x)\leq L,\qquad x\in\bR^n,k\in\Zn,$$
with $\mathbf{1}_A$ denoting the characteristic function of the set $A$.
The discrete mixed-norm $\alpha$-modulation space $m_{\vec{p},q}^{s,\alpha}$ is then defined as follows, see also \cite{MR4082240},
\begin{definition}
Let $\alpha\in [0,1)$, $s\in\bR$, $\vec{p}\in (0,\infty]^n$, and $q\in (0,\infty]$.
The discrete mixed-norm $\alpha$-modulation space $m_{\vec{p},q}^{s,\alpha}:=m_{\vec{p},q}^{s,\alpha}(\Zn\times\Zn)$ is defined a the set of all complex-valued sequences  $b=\{b_{j,m}\}_{u,m\in\Zn}$ such that
$$\|b\|_{m_{\vec{p},q}^{s,\alpha}}:=
\bigg(\sum_{j\in\Zn}\bigg\|\sum_{m\in\Zn} r_j^{s+\frac{ n}{2}}|b_{j,m}|\,\mathbf{1}_{Q(j,m)}\bigg\|_{\vec{p}}^q\bigg)^{1/q}<+\infty,$$
with $r_j$ given in Eq.\ \eqref{eq:rk}.
\end{definition}

Let us consider a simple tight frame with time-frequency properties adapted to the mixed-norm $\alpha$-modulation space $M_{\vec{p},q}^{s,\alpha}$. We mention that this approach to frame constructions has been considered in several contexts before, see, e.g., \cite{MR2035296} and references therein. 

Let $$Q_k:=Q_k^\alpha:=R[r_k k,ar_k],\qquad k\in\Zn,$$
be an $\alpha$-cover of cubes with $r_k$ defined in \eqref{eq:rk} with $a\geq \max\{2c_1,\pi\sqrt{n}/2\}$, and where $c_1$ is the constant from Example \ref{ex:cov}. Consider the localized trigonometric system given by 
$$e_{k,\ell}(\xi):=(2\pi)^{-n/2}r_k^{-n/2}\mathbf{1}_{Q_0}(r_k^{-1}\xi-k)e^{i\frac{\pi}{a}\ell\cdot (r_k^{-1}\xi-k)},\qquad k,\ell\in\Zn.$$
Then we define the system $\Phi:=\{\varphi_{k,\ell}\}_{k,\ell}$ in the Fourier domain by
\begin{equation}\label{eq:TF}
\hat{\varphi}_{k,\ell}(\xi)=\theta_k^\alpha(\xi)e_{k,\ell}(\xi),
\end{equation}
where $\{\theta_k^\alpha\}_k$ is the system given by Eq.\ \eqref{eq:theta}.
One can verify that
\begin{equation}\label{eq:muk}
\varphi_{k,\ell}(x)=(2a)^{-n/2}r_k^{n/2}e^{i r_k k\cdot x}\mu_k\bigg(\frac{\pi}{a}\ell+r_k x\bigg),   
\end{equation}
with the functions $\{\mu_k\}$ uniformly well-localised in the sense that for every $N\in\bN$ there exists $C_N<\infty$, independent of $k$, such that
$$|\mu_k(x)|\leq C_N\big(1+r_k|x|\big)^{-N},\qquad x\in\Rn.$$
This in turn implies that
\begin{equation}\label{eq:e1}
    |\varphi_{k,\ell}(x)|\leq C_N (2a)^{-n/2}r_k^{n/2}\big(1+r_k\big|x-x_{k,\ell}\big|\big)^{-N},\qquad x\in\Rn,
\end{equation}
with $x_{k,\ell}:=\frac{\pi}{a}r_k^{-1}\ell$.
In the frequency domain, we have $\text{supp}(\hat{\varphi}_{k,\ell})\subset B(\xi_k,cr_k)$ for some $c>0$ independent of $k$, which clearly implies that there exist constants $K_N$ such that the localisation
\begin{equation}\label{eq:e2}
   |\hat{\varphi}_{k,\ell}(\xi)|\leq K_N 
   r_k^{-\frac{n}{2}}(1+r_k^{-1}|\xi_{k}-\xi|)^{-N}, \qquad \xi\in\bR^n,
\end{equation}
holds true for $N\in\bN$. It can be verified, see \cite{Borup2007}, that $\Phi:=\{\varphi_{k,\ell}\}_{k,\ell}$ forms a tight frame for $L^2(\Rn)$, i.e., we have the identity
\begin{equation}\label{eq:tf}
f=\sum_{k,\ell} \langle f,\varphi_{k,\ell}\rangle \varphi_{k,\ell},\qquad f\in L^2(\Rn),
\end{equation}
where the sum converges unconditionally in $L^2(\Rn)$.

Let $\bC^{\Zn\times\Zn}$ denote the collection of all complex-valued sequences defined on $\Zn\times\Zn$. For $\beta\in\bR$ we define $\mathcal{C}_\beta$ to be the collection of $s\in \bC^{\Zn\times\Zn}$ satisfying
$$\|b\|_{\mathcal{C_\beta}}:=\sup_{k,\ell\in\Zn} \langle{k}\rangle^{-\beta}|b_{k,\ell}|<\infty,$$
and put $\mathcal{C}:=\bigcup_{\beta\in \bR} \mathcal{C}_\beta$.

The $\varphi$-transform, $S_\varphi$, is the operator defined by
$$S_\varphi:\mathcal{S}'\rightarrow \bC^{\Zn\times\Zn},\, f\mapsto S_\varphi f:=\{\langle f,\varphi_{k,\ell}\rangle\}_{k,\ell\in\Zn}$$

The inverse $\varphi$-transform $T_\phi$ is defined by
$$T_\varphi:\mathcal{C}\rightarrow \mathcal{S}' ,\, b\mapsto T_\varphi b:=\sum_{k,\ell} b_{k,\ell} \varphi_{k,\ell}.$$
It can be shown that $T_\varphi$ is in fact well-defined within the current setting,  we refer to  \cite{MR4082240} for further details. The following fundamental result on the $\varphi$-transform  is also proven in  \cite{MR4082240}. 
\begin{theorem}\label{th:phi}
Let $\alpha\in [0,1)$, $s\in\bR$, $\vec{p}\in(0,\infty)^n$, and $q\in (0,\infty)$. Then 
\begin{itemize}
    \item The $\varphi$-transform $S_\varphi$ is bounded from  $M^{s,\alpha}_{\vec{p},q}$ to $m_{\vec{p},q}^{s,\alpha}$
    \item The inverse $\varphi$-transform $T_\varphi$ is bounded from $m_{\vec{p},q}^{s,\alpha}$ to $M^{s,\alpha}_{\vec{p},q}$
    \end{itemize}
    with $T_\varphi\circ S_\varphi=\text{id}_{M^{s,\alpha}_{\vec{p},q}}$.
\end{theorem}

\section{Well-localized molecules and change of frame operators}\label{sec:alm}
One potential issue with the $\phi$-transform defined in Section \ref{sec:ms} is that the definition relies on one fixed reference frame $\Phi$. It could be that the construction only works with this type of calibration and, e.g., the construction is somehow unstable under small perturbations or fails for frames build with less smoothness. One of the main contributions of the present paper is to prove such concerns unfounded.

We will now focus on the time-frequency localization  specified by equations \eqref{eq:e1} and \eqref{eq:e2} rather than other  specific properties of the system $\Phi$. The estimates will be used in the sequel to define a general notion of  well-localised "molecules" for the transforms considered in this paper.

\begin{definition}\label{def:mole}
Let $\Psi:=\{\psi_{k,\ell}\}_{k,\ell\in\Zn}\subset L^2(\bR^n)$ be a family of complex-valued functions defined on $\bR^n$. We call $\Psi$ a family of $(M,N)$-molecules provided there exist finite constants $C_M, K_N$ such that uniformly in $k$ and $\ell$,
\begin{equation}\label{eq:ee1}
    |\psi_{k,\ell}(x)|\leq C_M (2a)^{-n/2}r_k^{n/2}\big(1+r_k\big|x-x_{k,\ell}\big|\big)^{-M},\qquad x\in\Rn,
\end{equation}
and
\begin{equation}\label{eq:ee2}
   |\hat{\psi}_{k,\ell}(\xi)|\leq K_N 
   r_k^{-\frac{n}{2}}(1+r_k^{-1}|\xi_{k}-\xi|)^{-N}, \qquad \xi\in\bR^n,
\end{equation}
with $r_k$, $\xi_k$, and $x_{k,\ell}$ given in Eq.\ \eqref{eq:rk}.
\end{definition}
Let us first prove that ''change-of-frame'' matrices between two families of $(M,N)$-molecules behave in a suitable sense as almost diagonal matrices, where we first make a few essential observations about the weight $\brac{\cdot}^\alpha$. 

It is a  well-known fact that there exists $C\geq 1$ such that $\brac{\xi+\zeta}\leq C\brac{\xi}\brac{\zeta}$, $\xi,\zeta\in\bR^n$. This estimate leads directly to the following observations about the weight $h_\alpha(\cdot):=\brac{\cdot}^\alpha$, $\alpha\in [0,1)$.
There exists $\beta,R_0,R_1,\rho_0,\rho_1>0$ such that $h_\alpha^{1+\beta}$ is moderate in the sense that
$h_\alpha^{1+\beta}(\cdot)\leq C\brac{\cdot}$, and
\begin{equation}\label{eq:b}
|\xi-\zeta|\leq \rho_0h_\alpha^{1+\beta}(\xi) \text{ implies that }R_0^{-1}\leq h_\alpha^{1+\beta}(\zeta)/h_\alpha^{1+\beta}(\xi)\leq R_0.
    \end{equation}

Also, we have
\begin{equation}\label{strawberry}
|\xi-\zeta| \le a h_\alpha(\xi) \text{ for }a\ge \rho_1 \text{ implies }h_\alpha(\zeta)\le R_1 a h_\alpha(\xi).
\end{equation}
Basen on these observations, we can now derive the following general matrix estimate.
\begin{lemma}\label{bubbleboy2}%
Let $\alpha\in [0,1)$ and let $\beta:=\beta(\alpha)>0$ satisfy \eqref{eq:b}. 
Choose $N,M,L>0$ such that $2N>n$ and $2M+2\tfrac{L}{\beta}>n$. If $\{\eta_{k,\ell}\}_{k,\ell\in\Zn}$ is a family of $\big(2N,2M+2\frac{L}{\beta}\big)$-molecules, and and $\{\psi_{j,m}\}_{j,m\in\Zn}$ is another family of $\big(2N,2M+2\frac{L}{\beta}\big)$-molecules, both in the sense of Definition \ref{def:mole}, then we have
\begin{align*}
|\langle\eta_{k,\ell},\psi_{j,m}\rangle|
\le& C
\min\bigg(\frac{r_k}{r_j},\frac{r_j}{r_k}\bigg)^{\frac{n}{2}+L}(1+\max(r_k,r_j)^{-1}|\xi_k-\xi_j|)^{-M}\\
&\phantom{Cspace} \times(1+\min(r_k,r_j)|x_{k,\ell}-x_{j,m}|)^{-N},
\end{align*}
with $r_k$, $\xi_k$, and $x_{n,\ell}$ defined in Eq.\ \eqref{eq:rk}.
\end{lemma}
\begin{proof}
From Lemma \ref{lastcall1} we have
\begin{align}\label{west}
|\langle\eta_{k,\ell},\psi_{j,m}\rangle|\le C
\min\bigg(\frac{r_k}{r_j},\frac{r_j}{r_k}\bigg)^{\frac{n}{2}}(1+\min(r_k,r_j)|x_{k,\ell}-x_{j,m}|)^{-2N}.
\end{align}
Using Lemma \ref{lastcall1} for
$\langle\hat{\eta}_{k,\ell},\hat{\psi}_{j,m}\rangle$ gives
\begin{equation}\label{potens}
|\langle\hat{\eta}_{k,\ell},\hat{\psi}_{j,m}\rangle|\le C
\min\bigg(\frac{r_k}{r_j},\frac{r_j}{r_k}\bigg)^{\frac{n}{2}}
(1+\max(r_k,r_j)^{-1}|\xi_k-\xi_j|)^{-2M-2\frac{L}{\beta}}.
\end{equation}
Next we raise the power of
the first term in \eqref{potens} at the expense of the second term.
Without loss of generality assume that $r_k\le r_j$. We first consider the
case $|\xi_k-\xi_j|\le \rho_0 r_j^{1+\beta}$, and use that
$h^{1+\beta}$ is moderate, see Eq.\ \eqref{eq:b}, to get
\begin{equation*}
\frac{1}{1+r_j^{-1}|\xi_k-\xi_j|}\le 1\le
R_0^{\frac{\beta}{1+\beta}}\bigg(\frac{r_k}{r_j}\bigg)^\beta.
\end{equation*}
In the other case, $|\xi_k-\xi_j|> \rho_0 r_j^{1+\beta}$, and it
follows by using the fact that  $r_k\ge \varepsilon_0$ for some $\varepsilon_0>0$, that
\begin{align*}
\frac{1}{1+r_j^{-1}|\xi_k-\xi_j|}\le
\frac{1}{\rho_0\varepsilon_0^\beta} \bigg(\frac{r_k}{r_j}\bigg)^\beta.
\end{align*}
Hence we have
\begin{align}\label{potens2}
|\langle\hat{\eta}_{k,\ell},\hat{\psi}_{j,m}\rangle|\le C
\min\bigg(\frac{r_k}{r_j},\frac{r_j}{r_k}\bigg)^{\frac{n}{2}+2L}
(1+\max(r_k,r_j)^{-1}|\xi_k-\xi_j|)^{-2M}.
\end{align}
The lemma follows by combining \eqref{west} and \eqref{potens2}, and using
\begin{equation*}
|\langle \eta_{k,\ell},\psi_{j,m}\rangle|=|\langle \eta_{k,\ell},\psi_{j,m}\rangle|^{\frac{1}{2}}|\langle \hat{\eta}_{k,\ell},\hat{\psi}_{j,m}\rangle|^{\frac{1}{2}}.
\end{equation*}
\end{proof}

Lemma \ref{bubbleboy2} provides a fundamental matrix estimate for well-localized molecules, but the result is only really useful if one can lift the discrete estimates to the continuous setting to obtain estimate for the mixed-norm $\alpha$-modulation spaces. We address this issue in the following section, where we also introduce a suitable algebra of almost diagonal matrices that is compatible with the estimate obtained in  Lemma \ref{bubbleboy2}.

\section{Matrix algebra of almost diagonal matrices}\label{sec:alge}
Inspired by the estimate in Lemma \ref{bubbleboy2}, we now define the following class of almost diagonal matrices.
\begin{definition}\label{doitmad}%
Assume that $s\in \bR$, $0<q< \infty$, and $\vec{p}\in (0,\infty)^n$. Let $$J:=\frac{n}{\min(1,q,p_1,p_2,\ldots,p_n)}.$$ A matrix $\mathbf{A}:=\{a_{(j,m)(k,n)}\}_{j,m,k,\ell\in\bZ^d}$
is called almost diagonal on $m^{s,\alpha}_{\vec{p},q}$ if
there exists $C, \delta>0$ such that
\begin{align*}
|a_{(j,m)(k,n)}|
&\le C \omega_{(j,m)(k,n)}, \qquad j,m,k,n\in \Zn,
\end{align*}
where
\begin{align*}
\omega_{(j,m)(k,n)}:=\bigg(\frac{r_k}{r_j}\bigg)^{s+\frac{n}{2}}
\min\bigg(\bigg(\frac{r_j}{r_k}\bigg)^{J+\frac{\delta}{2}},&\bigg(\frac{r_k}{r_j}\bigg)^{\frac{\delta}{2}}\bigg) c_{jk}^\delta\\
&\phantom{C}\times
(1+\min(r_k,r_j)|x_{k,n}-x_{j,m}|)^{-J-\delta},
\end{align*}
with
\begin{equation*}
c_{jk}^\delta:=\min\bigg(\bigg(\frac{r_j}{r_k}\bigg)^{J+\delta},\bigg(\frac{r_k}{r_j}\bigg)^{\delta}\bigg)
(1+\max(r_k,r_j)^{-1}|\xi_k-\xi_j|)^{-J-\delta}
\end{equation*}
with  $r_k$, $\xi_k$, and $x_{k,n}$ defined in Eq.\ \eqref{eq:rk}.
We denote the set of almost diagonal matrices on $m^{s,\alpha}_{\vec{p},q}$  by $\textrm{ad}_{\vec{p},q}^s$.
\end{definition}

\begin{remark}
Notice that $\textrm{ad}_{\vec{p},q}^s$ depends (in a somewhat weak sense) on the mixed-norm through the parameter $J$. 
\end{remark}

The matrix class considered in Definition \ref{doitmad} is an adaptation of a general matrix class for decomposition spaces considered by Rasmussen and the author in \cite{Nielsen2012}. In fact, it is proven in \cite{Nielsen2012} that such classes are closed under composition. To state the composition result more precisely, we introduce the quantities
\begin{align}
w^{s,\delta}_{(j,m)(k,n)}:=&\bigg(\frac{r_k}{r_j}\bigg)^{s+\frac{n}{2}}\min\bigg(\bigg(\frac{r_j}{r_k}\bigg)^{J+\frac{\delta}{2}},
\bigg(\frac{r_k}{r_j}\bigg)^{\frac{\delta}{2}}\bigg)c_{jk}^\delta \notag \\
& \phantom{C}\times
(1+\min(r_k,r_j)|x_{k,n}-x_{j,m}|)^{-J-\delta},\notag
\end{align}
where we have used the notation from Definition \ref{doitmad}. We then have
\begin{proposition}\label{poison}%
Let $s\in \bR$, $0<r\le 1$ and $\delta>0$. There exists $C>0$ such that
\begin{align*}
\sum_{i,l\in \Zn}
w^{s,\delta}_{(j,m)(i,l)}w^{s,\delta}_{(i,l)(k,n)}\le C
w^{s,\delta/2}_{(j,m)(k,n)}.
\end{align*}
\end{proposition}
We refer to \cite{Nielsen2012} for the proof of Proposition \ref{poison}. 
It follows from Proposition \ref{poison} that for $\delta_1,\delta_2>0$ we have
\begin{equation}\label{smellsliketeen}
\sum_{i,l\in \Zn}
w^{s,\delta_1}_{(j,m)(i,l)}w^{s,\delta_2}_{(i,l)(k,n)}\le C
w^{s,\min(\delta_1,\delta_2)/2}_{(j,m)(k,n)}
\end{equation}
which clearly proves that $\textrm{ad}_{\vec{p},q}^s$ is closed under composition.

As mentioned above, the class $\textrm{ad}_{\vec{p},q}^s$ depends only in a weak sense on the mixed-norm through the parameter $J$. However, the specific mixed-norm will play a much more central role when we consider the action of almost diagonal matrices on the discrete spaces $m^{s,\alpha}_{\vec{p},q}$. 

The following Proposition will show that  the almost diagonal matrices of Definition  \ref{doitmad} will act boundedly on $m^{s,\alpha}_{\vec{p},q}$.

\begin{proposition}\label{johnlennon}%
Assume that $s\in \bR$, $0\leq \alpha\leq 1$, $0<q< \infty$, and $\vec{p}\in (0,\infty)^n$, and suppose that $\mathbf{A} \in \textrm{ad}_{\vec{p},q}^s$. Then $\mathbf{A}$ is bounded on
$m^{s,\alpha}_{\vec{p},q}$.
\end{proposition}
\begin{proof}
 Let $s:=\{s_{k,\ell}\}_{k,\ell\in\Zd}\in
m^{s,\alpha}_{\vec{p},q}$, and put $r:=\min\{1,q,p_1,\ldots p_n\}$. We write
$\mathbf{A}:=\mathbf{A}_0+\mathbf{A}_1$ such that
\begin{equation*}
(\mathbf{A}_0s)_{(j,m)}\!=\!\!\sum_{k:r_k\ge r_j}\sum_{\ell\in\Zd}
a_{(j,m)(k,\ell)}s_{k,\ell}\hspace{0.2cm}\textrm{and}\hspace{0.2cm}(\mathbf{A}_1s)_{(j,m)}\!=\!\!\sum_{k:r_k<
r_j}\sum_{\ell\in\Zn} a_{(j,m)(k,\ell)}s_{k,\ell}.
\end{equation*}
By using Lemma \ref{le:max} we have
\begin{align*}
|(\mathbf{A}_0s)_{(j,m)}|&\le C\sum_{k:r_k\ge
r_j}\left(\frac{r_k}{r_j}\right)^{s+\frac{n}{2}-\frac{n}{r}-\frac{\delta}{2}}
c_{jk}^\delta
\sum_{\ell\in\Zn}\frac{|s_{k,\ell}|}{\left(1+r_j\left|x_{k,\ell}-x_{j,m}\right|\right)^{\frac{n}{r}+\delta}}\\
&\le C \sum_{k:r_k\ge
r_j}\left(\frac{r_k}{r_j}\right)^{s+\frac{n}{2}-\frac{\delta}{2}}
c_{jk}^\delta {\mathcal M}_r\Big(\sum_{\ell\in
\Zn}|s_{k,\ell}|\,\mathbf{1}_{Q(k,\ell)}\Big)(x),
\end{align*}
for $x\in Q(j,m)$. It then follows  that
\begin{align*}
\sum_{m\in\Zn} |(\mathbf{A}_0&s)_{(j,m)}\mathbf{1}_{Q(j,m)}| \le
C\sum_{k:r_k\ge
r_j}\left(\frac{r_k}{r_j}\right)^{s+\frac{n}{2}}
c_{jk}^\delta {\mathcal M}_r\Big(\sum_{\ell\in
\Zn}|s_{k,\ell}|\,\mathbf{1}_{Q(k,\ell)}\Big).
\end{align*}
We obtain, using the maximal estimate \eqref{max}, and the subadditivity  of $\|\cdot\|_{\vec{p}}^r$ given in Eq.\ \eqref{suba}, 
\begin{align}
\left\|\mathbf{A}_0s\right\|_{m^{s,\alpha}_{\vec{p},q}} &\le
C\bigg(\sum_{j\in\Zn} \bigg\|\sum_{k:r_k\ge r_j}c_{jk}^\delta r_k^{s+\frac{n}{2}} {\mathcal M}_r\Big(\sum_{\ell\in
\Zn}|s_{k,\ell}|\,\mathbf{1}_{Q(k,\ell)}\Big)\bigg\|_{\vec{p}}^q\bigg)^{1/q}\nonumber\\
&\le 
C\bigg(\sum_{j\in\Zn} \bigg\|\sum_{k:r_k\ge r_j}c_{jk}^\delta r_k^{s+\frac{n}{2}} \sum_{\ell\in
\Zn}|s_{k,\ell}|\,\mathbf{1}_{Q(k,\ell)}\bigg\|_{\vec{p}}^q\bigg)^{1/q}\nonumber\\
&= 
C\bigg(\sum_{j\in\Zn} \bigg\|\sum_{k:r_k\ge r_j}c_{jk}^\delta r_k^{s+\frac{n}{2}} \sum_{\ell\in
\Zn}|s_{k,\ell}|\,\mathbf{1}_{Q(k,\ell)}\bigg\|_{\vec{p}}^{r\cdot\frac{q}{r}}\bigg)^{1/q}\nonumber\\
&\le
C\bigg(\sum_{j\in\Zn} \bigg[\sum_{k:r_k\ge r_j}c_{jk}^{\delta r}\bigg\| r_k^{s+\frac{n}{2}} \sum_{\ell\in
\Zn}|s_{k,\ell}|\,\mathbf{1}_{Q(k,\ell)}\bigg\|_{\vec{p}}^{r}\bigg]^{\frac{q}{r}}\bigg)^{1/q}.\label{eq:big}
\end{align}
In the case $r<q$, we apply  H\"older's  inequality to the inner sum over $k$ in the estimate \eqref{eq:big}  to obtain
\begin{align*}\left\|\mathbf{A}_0s\right\|_{m^{s,\alpha}_{\vec{p},q}}&\le
C\bigg(\sum_{j\in\Zn} \bigg(\sum_{k:r_k\ge r_j}c_{jk}^{\delta r\frac{q}{q-r}}\bigg)^{\frac{q-r}{q}}\bigg[\sum_{k:r_k\ge r_j}c_{jk}^{\delta q}\bigg\| r_k^{s+\frac{n}{2}} \sum_{\ell\in
\Zn}|s_{k,\ell}|\,\mathbf{1}_{Q(k,\ell)}\bigg\|_{\vec{p}}^{q}\bigg]\bigg)^{1/q}\\
&\le
C\bigg(\sum_{j\in\Zn} \bigg[\sum_{k:r_k\ge r_j}c_{jk}^{\delta q}\bigg\| r_k^{s+\frac{n}{2}} \sum_{\ell\in
\Zn}|s_{k,\ell}|\,\mathbf{1}_{Q(k,\ell)}\bigg\|_{\vec{p}}^{q}\bigg]\bigg)^{1/q}\\
&\le
C\bigg(\sum_{k\in\Zn} \bigg\| r_k^{s+\frac{n}{2}} \sum_{\ell\in
\Zn}|s_{k,\ell}|\,\mathbf{1}_{Q(k,\ell)}\bigg\|_{\vec{p}}^{q}\bigg)^{1/q}\\
&=C\left\|s\right\|_{m^{s,\alpha}_{\vec{p},q}},
\end{align*}
where we have used  Lemma \ref{le:sum}. When $r=q\leq 1$, we may obtain the same estimate directly from \eqref{eq:big} using the subadditivity property \eqref{suba} with $r=q$.

The estimate of $\left\|\mathbf{A}_1s\right\|_{m^{s,\alpha}_{\vec{p},q}}$ follows from similar arguments and we leave the details  for the reader.
\end{proof}

\section{Some applications of the matrix algebra}\label{sec:appl}

Let us consider an application of Proposition \ref{johnlennon} to study various aspects of stability  of the $\varphi$-transform under perturbations. Suppose we fix $s$, $\alpha$, $\vec{p}$ and $q$, and consider $f\in M^{s,\alpha}_{\vec{p},q}(\bR^n)$. 
Let $\{\psi_{k,n}\}_{k,n\in\Zn}\subset L_2(\bR^n)$ be a system that satisfies, for some $C_1,C_2,\delta>0$ independent of $k$ and $n$,
\begin{align}
&|\psi_{k,n}(x)|\le C_1 t_k^{\frac{\nu}{2}}(1+t_k|x_{k,n}-x|)^{-2\left(J+\delta\right)}\label{bruce},\\
&|\hat{\psi}_{k,n}(\xi)|\le C_2
t_k^{-\frac{\nu}{2}}(1+t_k^{-1}|\xi_{k}-\xi|)^{-2\left(J+\delta\right)-
\frac{2}{\beta}\left(|s|+2J+\frac{3\delta}{2}\right)},\label{bspringsteen}
\end{align}
with  $r_k$, $\xi_k$, and $x_{k,\ell}$ defined in Eq.\ \eqref{eq:rk}, and with $J$ is defined in  Definition \ref{doitmad}.
Notice that we are {\em not} assuming that $\{\psi_{k,n}\}_{k,n\in\Zn}\subset \mathcal{S}(\bR^n)$, so we need to give meaning to the pairing $\langle f,\psi_{j,m}\rangle$.

We first notice that 
\begin{equation}\label{eq:as}
    \{\langle \varphi_{k,n},\psi_{j,m}\rangle\}\in \textrm{ad}_{\vec{p},q}^s,
\end{equation}
which follows from  using the estimates \eqref{bruce} and \eqref{bspringsteen} in Lemma \ref{bubbleboy2}.
Motivated by the fact that $\{\varphi_{k,n}\}_{k,n\in\Zn}$ is a tight frame for $L_2(\bR^n)$, see Eq.\ \eqref{eq:tf}, we formally define $\langle f,\psi_{j,m}\rangle$ as
\begin{equation}
\langle f,\psi_{j,m}\rangle:=\sum_{k,n\in \Zn}\langle \varphi_{k,n},\psi_{j,m}\rangle \langle
f,\varphi_{k,n}\rangle, \,\, f \in M^{s,\alpha}_{\vec{p},q}.
\end{equation}
Proposition \ref{johnlennon}, together with the observation in \eqref{eq:as}, shows that $\langle \cdot,\psi_{j,m}\rangle$ is a bounded linear functional on $M^{s,\alpha}_{\vec{p},q}$- More precisely, we have
\begin{align*}
\sum_{k,n\in\Zn}|\langle \varphi_{k,n},\psi_{j,m}\rangle||\langle
f,\varphi_{k,n}\rangle| \notag
&\le C\Big\|\Big\{\sum_{k,n\in\Zn}|\langle \varphi_{k,n},\psi_{j,m}\rangle| |\langle
f,\varphi_{k,n}\rangle| \Big\}_{j,m\in\Zn}\Big\|_{m^{s,\alpha}_{\vec{p},q}}\notag \\ &\le C' \|\langle
f,\varphi_{k,n}\rangle \|_{m^{s,\alpha}_{\vec{p},q}}\notag\\&\le
C''\|f\|_{M^{s,\alpha}_{\vec{p},q}},\notag
\end{align*}
where we used Theorem \ref{th:phi} for the final estimate. 
Furthermore, a similar argument shows that  $\{\psi_{j,m}\}_{j,m\in\Zn}$ is a norming family for $M^{s,\alpha}_{\vec{p},q}$ in the sense that  \begin{equation}\label{eq:norming} \|\langle f,\psi_{j,m}\rangle\|_{m^{s,\alpha}_{\vec{p},q}} \le C \|f\|_{M^{s,\alpha}_{\vec{p},q}},
\end{equation}
with $C$ independent of $f$. 

Notice that Eq.\ \eqref{eq:norming} provides an analog of the $\phi$-tranform for the system $\{\psi_{k,n}\}_{k,n\in\Zn}$. The observation in Eq.\ \eqref{eq:as} can also be used to obtain an analog of  the inverse $\phi$-transform for the same system. Suppose we have a finite sum
$$g:=\sum_{j,m} c_{j,m} \psi_{j,m}.$$
Then, using Theorem \ref{th:phi} and Proposition \ref{johnlennon},
\begin{align}
    \|g\|_{M^{s,\alpha}_{\vec{p},q}}\asymp \|\langle f,\varphi_{k,n}\rangle\|_{m^{s,\alpha}_{\vec{p},q}}
    \leq C  \|\langle \{c_{j,m}\}\|_{m^{s,\alpha}_{\vec{p},q}},\label{eq:TT}
\end{align}
since 
\begin{equation}
\langle g,\varphi_{k,n}\rangle:=\sum_{j,m\in \Zn}\langle \psi_{j,m},\varphi_{k,n}\rangle c_{j,m}.
\end{equation}
One can then extends the validity of \eqref{eq:TT} to any  $\{c_{j,m}\}\in {m^{s,\alpha}_{\vec{p},q}} $ by continuity.

\subsection{Frames obtained by perturbations}
We now consider a system $\{\psi_{k,n}\}_{k,n\in\Zn}\subset L_2(\bR^n)$ that is close to the $\varphi$-transform tight frame $\Phi:=\{\varphi_{k,n}\}_{k,n\in\Zn}$ in
the sense that there exists $\varepsilon, \delta, C_1,C_2,\delta>0$ independent of $k$ and $n$, such that
\begin{align}
&|\varphi_{k,n}(x)-\psi_{k,n}(x)|\le C_1 \varepsilon t_k^{\frac{\nu}{2}}(1+t_k|x_{k,n}-x|)^{-2\left(J+\delta\right)}\label{bruce4},\\
&|\hat{\varphi}_{k,n}(\xi)-\hat{\psi}_{k,n}(\xi)|\le C_2 \varepsilon
t_k^{-\frac{\nu}{2}}(1+t_k^{-1}|\xi_{k}-\xi|)^{-2\left(J+\delta\right)-
\frac{2}{\beta}\left(|s|+2J+\frac{3\delta}{2}\right)},\label{bspringsteen4}
\end{align}
with  $r_k$, $\xi_k$, and $x_{k,\ell}$ defined in Eq.\ \eqref{eq:rk}, and with $J$ is defined in  Definition \ref{doitmad}. 

Using a technique introduced by Kyriasiz and Petrushev \cite{MR2204289}, one obtains the following stability result. We have included the proof for the convenience of the reader.

\begin{proposition}\label{generation1}%
There exists $\varepsilon_0,C_1,C_2>0$ such that if $\{\psi_{k,n}\}_{k,n\in\Zn}$ satisfies \eqref{bruce4} and
\eqref{bspringsteen4} for some $0<\varepsilon \le \varepsilon_0$ and $f\in M^{s,\alpha}_{\vec{p},q}(\Rn)$, then we have
\begin{equation}
C_1\|f\|_{M^{s,\alpha}_{\vec{p},q}}\le \|\langle
f,\psi_{k,n}\rangle\|_{m^{s,\alpha}_{\vec{p},q}}\le C_2\|f\|_{M^{s,\alpha}_{\vec{p},q}}.
\end{equation}
\end{proposition}
\begin{proof}
 We have already observed that  $\{\psi_{k,n}\}_{k,n\in\Zn}$ is a norming
family, which gives the upper bound. Thus we only need to establish the
lower bound. For this we notice that
$\{\varepsilon^{-1}(\varphi_{k,n}-\psi_{k,n})\}_{k,n\in\Zn}$ is also a norming family
so we have
\begin{equation*}
\|\langle f,\varphi_{k,n}-\psi_{k,n}\rangle\|_{m^{s,\alpha}_{\vec{p},q}}\le C
\varepsilon \|f\|_{M^{s,\alpha}_{\vec{p},q}}.
\end{equation*}
It then follows from  Theorem \ref{th:phi} that
\begin{align*}
\|f\|_{M^{s,\alpha}_{\vec{p},q}} &\le C \|\langle
f,\varphi_{k,n}\rangle\|_{m^{s,\alpha}_{\vec{p},q}}\\
&\le C' (\|\langle f,\psi_{k,n}\rangle\|_{m^{s,\alpha}_{\vec{p},q}}+\|\langle
f,\varphi_{k,n}-\psi_{k,n}\rangle\|_{m^{s,\alpha}_{\vec{p},q}}) \\
&\le C''(\|\langle f,\psi_{k,n}\rangle\|_{m^{s,\alpha}_{\vec{p},q}}+\varepsilon
\|f\|_{M^{s,\alpha}_{\vec{p},q}}).
\end{align*}
By choosing $\varepsilon<1/C''$ we get the lower bound. \\
\end{proof}

One can use the freedom provided by the slack parameter $\varepsilon_0>0$ in Proposition \ref{generation1} to impose various desirable properties on the system $\{\psi_{k,n}\}_{k,n\in\Zn}$.   As an example, we will construct frames with compact support in Section \ref{sec:csupp}.

By combining Propositions \ref{johnlennon} and \ref{generation1}, we can also deduce various boundedness properties of  the frame
operator
\begin{equation*}
 S f=\sum_{k,n\in\Zn}\langle f,\psi_{k,n}\rangle
\psi_{k,n}.
\end{equation*}
In fact, notice that
\begin{equation*}
 \langle S f,\varphi_{j,m}\rangle =\sum_{k,n\in\Zn}\langle f,\psi_{k,n}\rangle
\langle \psi_{k,n},\varphi_{j,m}\rangle,
\end{equation*}
and since we have the almost diagonal estimate
$$\{\langle \psi_{k,n},\varphi_{j,m}\rangle\}\in \textrm{ad}_{\vec{p},q}^s,$$
one obtains $\{ \langle S f,\varphi_{j,m}\rangle\}_{j,m}\in m^{s,\alpha}_{\vec{p},q}$. Using Theorem \ref{th:phi} this implies that $$S:M^{s,\alpha}_{\vec{p},q}\rightarrow M^{s,\alpha}_{\vec{p},q}.$$

 In case the operator $S$ is invertible on $M^{s,\alpha}_{\vec{p},q}$, one can easily obtain a corresponding norm convergent frame expansion 
 \begin{equation}\label{eq:fra}
 f=SS^{-1} f= \sum_{j,m} \langle S^{-1}f,\psi_{j,m}\rangle \psi_{j,m}, \qquad f\in M^{s,\alpha}_{\vec{p},q}.     \end{equation}
 The question on invertibility of $S$ is addressed in the following Proposition \ref{layla} . We have included the proof of Proposition \ref{layla} for the convenience of the reader, but mention that the proof follows the same lines as in  \cite{MR2204289}. 
\begin{proposition}\label{layla}%
There exists $\varepsilon_0>0$ such that if $\{\psi_{k,n}\}_{k,n\in\Zn}$ is a frame for $L_2(\Rn)$
and satisfies \eqref{bruce4} and \eqref{bspringsteen4} for some $0<\varepsilon\le\varepsilon_0$, then $S$ is boundedly invertible on $M^{s,\alpha}_{\vec{p},q}$.
\end{proposition}
\begin{proof}
 The fact that $\{\psi_{k,n}\}_{k,n\in\Zn}$ is a
frame for $L_2(\Rn)$ ensures that $S^{-1}$ exists as a bounded operator on $L_2(\Rn)$. To verify that $S^{-1}$ is bounded on $M^{s,\alpha}_{\vec{p},q}(\Rn)$, it suffices to show that
\begin{equation}\label{americanpie}
\|(I-S)f\|_{M^{s,\alpha}_{\vec{p},q}}\le C\varepsilon\|f\|_{M^{s,\alpha}_{\vec{p},q}}, \qquad f\in M^{s,\alpha}_{\vec{p},q},
\end{equation}
choosing $\varepsilon$ small enough and using a standard Neumann series approach. Assume
for the moment that $\mathcal{D}:=\{d_{(j,m)(k,n)}\}:=\{\langle (I-S)\varphi_{k,n},
\varphi_{j,m}\rangle\}$ satisfies
\begin{equation}\label{wondersteve}
\|\mathcal{D}s\|_{m^{s,\alpha}_{\vec{p},q}}\le C\varepsilon\|s\|_{m^{s,\alpha}_{\vec{p},q}}.
\end{equation}
By using that $\{\psi_{k,n}\}_{k,n\in\Zn}$ is a frame for $L_2(\Rn)$, we have that $S$ is self-adjoint which leads to
\begin{align*}
\|(I-S)f\|_{M^{s,\alpha}_{\vec{p},q}}&\le
C\|\langle(I-S)f,\varphi_{j,m}\rangle\|_{m^{s,\alpha}_{\vec{p},q}}=C\|\mathcal{D}\{\langle
f,\varphi_{k,n}\rangle\}_{k,n\in\Zn}\|_{m^{s,\alpha}_{\vec{p},q}}\\ &\le C\varepsilon\|\langle
f,\varphi_{j,m}\rangle\|_{m^{s,\alpha}_{\vec{p},q}}\le C\varepsilon
\|f\|_{M^{s,\alpha}_{\vec{p},q}},
\end{align*}
where we have used Proposition \ref{johnlennon} and Theorem \ref{th:phi}.
So to show \eqref{americanpie} it suffices to prove
\eqref{wondersteve}. 
We obtain the  decomposition
\begin{equation*}
\mathcal{D}=\mathcal{D}_1\mathcal{D}_2+\mathcal{D}_3\mathcal{D}_4,
\end{equation*}
with
\begin{align*}
&\mathcal{D}_1:=\{d_{1(j,m)(i,l)}\}:=\{\langle\varphi_{i,l}-\psi_{i,l},\varphi_{j,m}\rangle\},\\
&\mathcal{D}_2:=\{d_{2(i,l)(k,n)}\}:=\{\langle\varphi_{k,n},\varphi_{i,l}\rangle\},\\
&\mathcal{D}_3:=\{d_{3(j,m)(i,l)}\}:=\{\langle\psi_{i,l},\varphi_{j,m}\rangle\},\\
&\mathcal{D}_4:=\{d_{4(i,l)(k,n)}\}:=\{\langle\varphi_{k,n},\varphi_{i,l}-\psi_{i,l}\rangle\},
\end{align*}
by observing that 
\begin{align*}
\langle(I-S)\varphi_{k,n},\varphi_{j,m}\rangle&=\sum_{i,l\in\Zn}\langle\varphi_{k,n},\varphi_{i,l}\rangle\langle\varphi_{i,l},\varphi_{j,m}\rangle
-\sum_{i,l\in\Zn}\langle\varphi_{k,n},\psi_{i,l}\rangle\langle\psi_{i,l},\varphi_{j,m}\rangle\\
&=\sum_{i,l\in\Zn}\langle\varphi_{k,n},\varphi_{i,l}\rangle\langle\varphi_{i,l}-\psi_{i,l},\varphi_{j,m}\rangle
+\sum_{i,l\in\Zn}\langle\varphi_{k,n},\varphi_{i,l}-\psi_{i,l}\rangle\langle\psi_{i,l},\varphi_{j,m}\rangle.
\end{align*}
Since $\{\psi_{k,n}\}_{k,n\in\Zn}$ satisfies \eqref{bruce4} and
\eqref{bspringsteen4}, we deduce from Lemma \ref{bubbleboy2} that
$$\varepsilon^{-1}\mathcal{D}_1,\mathcal{D}_2,\mathcal{D}_3,
 \varepsilon^{-1}\mathcal{D}_4\in \textrm{ad}_{\vec{p},q}^s.$$ Next, we use that $\textrm{ad}_{\vec{p},q}^s$ is closed under composition, see Eq.\ \eqref{smellsliketeen}, and apply Proposition \ref{johnlennon} to obtain,
\begin{equation*}
\|\mathcal{D}s\|_{m^{s,\alpha}_{\vec{p},q}}\le C\varepsilon\|s\|_{m^{s,\alpha}_{\vec{p},q}}.
\end{equation*}
Consequently, \eqref{americanpie} holds,
and for sufficiently small $\varepsilon$, the operator $S$ is
boundedly invertible on $M^{s,\alpha}_{\vec{p},q}$. 
\end{proof}

\subsection{Compactly supported frames}\label{sec:csupp}
\indent For numerical purposes, it is often convenient to have access to discrete compactly supported frames for the smoothness spaces 
$M^{s,\alpha}_{\vec{p},q}$. The $\varphi$-transform consists of band-limited frame elements that cannot be compactly supported, but one can bypass this restriction by approximating the $\varphi$-transform elements by suitable compactly supported functions. In fact, it suffices to prove that there exists a system of functions $\{\tau_k\}_{k\in\Zn}\subset L_2(\Rn)$ which is close enough to the family $\{\mu_k\}_{k\in\Zn}$ given in \eqref{eq:muk}:
\begin{align*}
|\mu_k(x)-\tau_k(x)|&\le \varepsilon(1+|x|)^{-2\left(J+\delta\right)},\\
|\hat{\mu}_k(\xi)-\hat{\tau}_k(\xi)|&\le \varepsilon(1+|\xi|)^{-2\left(J+\delta\right)-\frac{2}{\beta}\left(|s|+{2J}+\frac{3\delta}{2}\right)}.
\end{align*}
 The system
\begin{equation}\label{eq:system}
\{\psi_{k,\ell}\}_{k,\ell\in\Zn}:=\Big\{r_k^{n/2}\tau_k\Big({r_k}x-\frac{\pi}{a}n\Big)e^{ix
\cdot \xi_k}\Big\}_{k,\ell\in\Zn}
\end{equation}
will then satisfy \eqref{bruce4} and \eqref{bspringsteen4}. First, we take $g\in C^{1}(\bR^n)\cap L_2(\bR^n)$, $\hat{g}(0)\not= 0$, which for fixed $N,M>0$ satisfies
\begin{align}
|g^{(\kappa)}(x)|&\le C(1+|x|)^{-N-1}, \,\,
|\kappa|\le 1, \label{kappastaff1} \\
|\hat{g}(\xi)|&\le C(1+|\xi|)^{-M-1}.
\label{kappastaff2}
\end{align}
Next for $m\ge 1$, we define $g_m(x):=C_g m^n g(m x)$,
where $C_g:=\hat{g}(0)^{-1}$. It then follows that
\begin{align}
|g_m^{(\kappa)}(x)|&\le C
m^{n+\alpha_2|\kappa|}(1+m|x|)^{-N-1},\,\,
 |\kappa|\le 1, \notag\\
&\phantom{C}\int_{\bR^n} g_m(x) \d x = 1 \label{gstar2},\\
|\hat{g}_m(\xi)|&\le C m^{M+\alpha_2}(1+|\xi|)^{-M-1}.\notag
\end{align}
To construct $\tau_k$ we also need the following set of $K$-term  linear combinations,
\begin{equation*}
\Theta_{K,m}=\{\psi : \psi(\cdot)=\sum_{i=1}^K a_i
g_m(\cdot+b_i),a_i\in \bC, b_i\in \Rn\}.
\end{equation*}
We can now call one the following general approximation result showing that any function with sufficient decay in both direct and frequency space can be approximated to an arbitrary degree by a finite linear combination of translates and dilates of another function with similar decay.
\begin{proposition}\label{notmylove}%
Let $N'>N>n$ and $M'>M>n$. If $g\in C^{1}(\Rn)\cap L_2(\Rn)$, $\hat{g}(0)\not= 0$, fulfills \eqref{kappastaff1} and \eqref{kappastaff2} and $\mu_k\in C^{1}(\Rn)\cap L_2(\Rn)$ fulfill
\begin{align*}
|\mu_k(x)|&\le C(1+|x|)^{-N'}, \\
|\mu_k^{(\kappa)}(x)|&\le C, \, |\kappa|\le 1,
\\ |\hat{\mu}_k(\xi)|&\le C (1+|\xi|)^{-M'},
\end{align*}
then for any $\varepsilon>0$ there exists $K,m \ge 1$ and $\tau_k \in
\Theta_{K,m}$ such that
\begin{align}
|\mu_k(x)-\tau_k(x)|&\le \varepsilon (1+|x|)^{-N}, \label{lostsomeone1} \\
|\hat{\mu}_k(\xi)-\hat{\tau}_k(\xi)|&\le \varepsilon (1+|\xi|)^{-M}. \label{lostsomeone2}
\end{align}
\end{proposition}
The proof of Proposition \ref{notmylove} can be found in \cite{Nielsen2012}. We can now summarise the results in the following Corollary.

\begin{corollary}\label{mualim1}%
Let $s\in \bR$, $0\leq \alpha\leq 1$, $0<q< \infty$, and $\vec{p}\in (0,\infty)^n$, and let $J$ be as in Definition \ref{doitmad}. 
Assume that $g\in C^{1}(\bR^n)\cap L_2(\bR^n)$, with $\hat{g}(0)\not= 0$, fulfills
\eqref{kappastaff1} and \eqref{kappastaff2} with $N=2\left(J+\delta\right)$ and $M=2\left(J+\delta\right)+\frac{2}{\beta}\left(|s|+2J+\frac{3\delta}{2}\right)$ for some $\delta>0$.
Then there exists $K,m \ge 1$ and $\tau_k \in
\Theta_{K,m}$ such that the system 
$\{\psi_{k,\ell}\}_{k,\ell\in\Zn}$ defined in Eq.\ \eqref{eq:system} is a frame for $M^{s,\alpha}_{\vec{p},q}$ in the sense that there exist constants $C_1$ and $C_2$ such that for 
$f\in M^{s,\alpha}_{\vec{p},q}$,
\begin{equation*}
C_1\|f\|_{M^{s,\alpha}_{\vec{p},q}}\le \|\langle
f,\psi_{k,\ell}\rangle\|_{m^{s,\alpha}_{\vec{p},q}}\le C_2\|f\|_{M^{s,\alpha}_{\vec{p},q}},
\end{equation*}
and
\begin{equation*}
f=\sum_{k,\ell\in\Zn}\langle f,S^{-1}\psi_{k,\ell}\rangle \psi_{k,\ell},
\end{equation*}
with norm convergence in $M^{s,\alpha}_{\vec{p},q}$.
\end{corollary}
\begin{proof}
It is easy to verify that $M^{0,\alpha}_{(2,2,\ldots,2),2}=L^2(\bR^n)$ and $m^{0,\alpha}_{(2,2,\ldots,2),2}=\ell^2$. Hence, using 
Proposition \ref{generation1} with $\vec{p}_0=(2,2,\ldots,2)$ and $q_0=2$, there exists $\varepsilon_0>0$ for which $0<\varepsilon\leq \varepsilon_0$ ensures that any system $\{\psi_{k,\ell}\}_{k,\ell\in\Zn}$ satisfying  \eqref{bruce4} and
\eqref{bspringsteen4} for $0<\varepsilon \le \varepsilon_0$ is a frame for $L^2(\Rn)$. By decreasing the value of $\epsilon_0$, if needed, we can assume that the conclusions of Propositions \ref{generation1} and \ref{layla} also holds for 
the given $\vec{p}$ and $q$ whenever $0<\varepsilon \le \varepsilon_0$. 

Pick some $0<\varepsilon \le \varepsilon_0$. Using Proposition \ref{notmylove}, we pick $\tau_k \in
\Theta_{K,m}$ such that Eqs.\ \eqref{lostsomeone1} and \eqref{lostsomeone2} are satisfied for $N=2\left(J+\delta\right)$ and $M=2\left(J+\delta\right)+\frac{2}{\beta}\left(|s|+2J+\frac{3\delta}{2}\right)$ for a chosen $\delta>0$. The system $\{\psi_{k,\ell}\}_{k,\ell\in\Zn}$ given in Eq.\ \eqref{eq:system} is thus a frame for $L^2(\Rn)$, with \begin{equation*}
 S f=\sum_{k,n\in\Zn}\langle f,\psi_{k,n}\rangle
\psi_{k,n}.
\end{equation*}
a self-adjoint invertible operator on $L^2(\Rn)$. We can now conclude using Propositions \ref{generation1} and \ref{layla} with the given $\vec{p}$ and $q$.
\end{proof}

\begin{remark}
For numerical purposes, it can be beneficial  to choose the function $g$ to be a B-spline with suitable smoothness as the reservoir $\Theta_{K,m}$ then consists of linear combinations of translated and dilated B-splines. This will e.g.\ facilitate fast multiscale numerical algorithms for approximating the inner products $\{\langle f,\psi_{k,n}\rangle\}$. 
\end{remark}

\subsection{Fourier multipliers}
Another useful application of classes of almost diagonal matrices is to study mapping properties of various linear operators between the associated smoothness spaces. In this section, we show how to use this approach to obtain relatively straightforward boundedness result for Fourier multipliers on mixed-norm $\alpha$-modulation spaces. We  mention that another (perhaps more direct) approach to Fourier multiplies on mixed-norm $\alpha$-modulation spaces is considered in  \cite[Section 6]{MR4082240}.

Let $m$ be a bounded measurable function on $\Rn$. The associated Fourier multiplier is the operator
$$m(D)f = \mathcal{F}^{-1}(m\hat{f}),$$
which is initially defined and bounded on $L^2(\Rn)$. In particular, $m(D)$ is defined on a dense subset of $M^{s,\alpha}_{\vec{p},q}(\Rn)$ provided $\vec{p}\in (0,\infty)^n$ and $0<q<\infty$. Let us fix $\alpha\in [0,1]$ and some $b\in\bR$. We assume that the multiplier function $m$ satisfies the smoothness condition
$$\sup_{\xi\in\Rn}\langle \xi\rangle^{\alpha|\eta|-b}|\partial^\eta m(\xi)|<\infty,$$
for every multi-index $\eta\in (\bN\cup\{0\})^n$. We claim that the following result holds,
\begin{equation}\label{eq:bdd}
m(D):M^{s+b,\alpha}_{\vec{p},q}(\Rn)\rightarrow M^{s,\alpha}_{\vec{p},q}(\Rn).    
\end{equation}

Calling on Theorem \ref{th:phi} and Proposition \ref{johnlennon}, it is straightforward to verify that it suffices to show that the matrix
\begin{equation}\label{eq:bdda}
 \left\{\big\langle \langle \xi_k\rangle^{-b} m(D) \varphi_{k,\ell},\varphi_{j,m}\big\rangle\right\}\in \textrm{ad}_{\vec{p},q}^s.
 \end{equation}

We first observe, using the compact support properties of the system $\{\phi_{k,\ell}\}$, that $\langle m(D)\varphi_{k,\ell},\varphi_{j,m}\rangle=0$ whenever $Q_k\cap Q_j=\emptyset$. Let us therefore focus on the case
$Q_k\cap Q_j\not=\emptyset$, where we have $r_k\asymp r_j$ (with constants independent of $k$ and $j$). The equivalence $r_k\asymp r_j$ in turn implies that, using the notation from Definition \ref{doitmad},
\begin{equation}\label{eq:bddb}
\omega_{(k,\ell)(j,m)}\asymp (1-|\ell-m|)^{-J-\delta}.
\end{equation}
We have, using Eq.\ \eqref{eq:TF},
\begin{align*}
    \langle m(D)\varphi_{k,\ell},\varphi_{j,m}\rangle&=\int_{\Rn} m(\xi)\theta_{k}^\alpha(\xi)\theta_{j}^\alpha(\xi)e_{k,\ell}(\xi)\overline{e_{j,m}}(\xi)\,d\xi,
\end{align*}
and by the affine change of variable $\xi:=T_ky:=r_k y+\xi_k$,
\begin{align*}
    \langle m(D)\varphi_{k,\ell},\varphi_{j,m}\rangle&=r_k^n\int_{\Rn} m(T_k y)\theta_{k}^\alpha(T_k y)\theta_{j}^\alpha(T_k y)e_{k,\ell}(T_k y)\overline{e_{j,m}}(T_k y)\,dy\\
    &=(2\pi)^{-n}\int_{\Rn} m(T_k y)\theta_{k}^\alpha(T_k y)\theta_{j}^\alpha(T_k y)\\&\phantom{(2\pi)^{-n}aaa}\times \exp\left[i\frac{\pi}{a}\bigg(\big(\ell-\frac{r_k}{r_j}m\big)\cdot y -\frac{r_k}{r_j}m\cdot k +m\cdot j\bigg)\right]\,dy.
\end{align*}
Hence, letting $g_{k,j}(y):=m(T_k y)\theta_{k}^\alpha(T_k y)\theta_{j}^\alpha(T_k y)$, we obtain
$$|\langle m(D)\varphi_{k,\ell},\varphi_{j,m}\rangle|\leq C\left|\cF[g_{k,j}]\bigg(\frac{\pi}{a}\bigg[\frac{r_k}{r_j}m-\ell\bigg]\bigg)\right|.$$
Now we proceed to make a standard decay estimate for the Fourier transform of $g_{k,j}$. Notice that
$\theta_{k}^\alpha(T_k \cdot)\theta_{j}^\alpha(T_k \cdot)$ is $C^\infty$ with support contained in a compact set $\Omega$ that can be chosen independent of $k$ and $j$, which can be verified using the fact that  $r_k\asymp r_j$. Hence, by the Leibniz rule, one obtains for $\beta\in (\bN\cup \{0\})^n$,
\begin{align*}|\partial^\beta [g_{k,j}](\xi)|&\leq C_\beta\mathbf{1}_\Omega(\xi)\sum_{\eta\leq \beta} \partial^\eta[m(T_k\cdot)](\xi),\qquad \xi\in \Rn.
\end{align*}
Then by the chain-rule, recalling that $r_k=\langle \xi_k\rangle^\alpha$,
\begin{align*}|\partial^\beta [g_{k,j}](\xi)|&\leq C_\beta \mathbf{1}_\Omega(\xi)\sum_{\eta\leq \beta} r_k^{|\eta|}[(\partial^\eta m)(T_k\cdot)](\xi)\\
&\leq C\sum_{\eta\leq \beta}r_k^{|\eta|} \mathbf{1}_\Omega(\xi) \langle T_k\xi\rangle^{b-\alpha |\eta|}\\
&\leq C\mathbf{1}_\Omega(\xi)\sum_{\eta\leq \beta}\langle \xi_k\rangle^{\alpha|\eta|} \langle \xi_k\rangle^{b-\alpha |\eta|},
\end{align*}
where we used that $r_k=\langle \xi_k\rangle^\alpha$.
Standard estimates now show that for any $N>0$, there exists $C_N<\infty$ such that
$$\langle \xi_k\rangle^{-b}|\langle m(D)\varphi_{k,\ell},\varphi_{j,m}\rangle|\leq C_N (1+|m-\ell|)^{-N},$$
whenever $Q_k\cap Q_j\not=\emptyset$. By the observation in Eq.\ \eqref{eq:bddb}, we may therefore conclude that \eqref{eq:bdda} holds, and consequently, that the multiplier result \eqref{eq:bdd} also holds.
\appendix

\section{Some technical lemmas}
\begin{lemma}\label{lastcall1}\noindent
Let $N>n$ and suppose the system $\{\eta_{k,\ell}\}_{k,\ell\in\Zn}$ satisfies \eqref{eq:ee1}, and
the system $\{\psi_{k,\ell}\}_{k,n\Zn}$ also satisfies \eqref{eq:ee1}. We then have
\begin{align}\label{west1aa}
|\langle\eta_{k,\ell},\psi_{j,m}\rangle|\le C
\min\bigg(\frac{r_k}{r_j},\frac{r_j}{r_k}\bigg)^{\frac{n}{2}}(1+\min(r_k,r_j)|x_{k,\ell}-x_{j,m}|)^{-N},
\end{align}
with $r_k$  and $x_{k,\ell}$ given in \eqref{eq:rk}.
\end{lemma}
\begin{proof}
Without loss of generality assume that $r_k\le r_j$. First we
consider the case $r_k|x_{k,m}-x_{j,m}|\le 1$. It then follows
that
\begin{equation}\label{dontwanna}
\frac{r_k^{\frac{n}{2}}}{(1+r_k|x_{k,\ell}-x|)^{N}}\le
r_k^{\frac{n}{2}}\le
\frac{2^{N}r_k^{\frac{n}{2}}}{(1+r_k|x_{k,\ell}-x_{j,m}|)^{N}},
\end{equation}
and we have
\begin{align}
|\langle\eta_{k,\ell},\psi_{j,m}\rangle| &\le
\frac{Cr_k^{\frac{n}{2}}}{(1+r_k|x_{k,\ell}-x_{j,m}|)^{N}}
\int_{\Rn}\frac{r_j^{\frac{n}{2}}}{(1+r_j|x_{j,m}-x|)^{N}}\d x \notag \\
&= \frac{Cr_k^{\frac{n}{2}}}{(1+r_k|x_{k,\ell}-x_{j,m}|)^{N}}\int_{\Rn}\frac{r_j^{-\frac{n}{2}}}{(1+|x|)^{N}}\d x \notag \\
&\le C
\bigg(\frac{r_k}{r_j}\bigg)^{\frac{n}{2}}(1+r_k|x_{k,\ell}-x_{j,m}|)^{-N}, \label{workdowntown}
\end{align}
where we used that $N>n$ to estimate the integral.
 For the other case, $r_k|x_{k,m}-x_{j,m}|> 1$, we consider two additional cases.
In the first case, we assume that $|x_{k,\ell}-x| \ge
\tfrac{1}{2}|x_{k,\ell}-x_{j,m}|$. Similar to above we then get
\eqref{dontwanna} which leads to \eqref{workdowntown}. In the last case, we have $|x_{k,\ell}-x| <
\tfrac{1}{2}|x_{k,\ell}-x_{j,m}|$ which gives $|x_{j,m}-x| >
\tfrac{1}{2}|x_{k,\ell}-x_{j,m}|$. It then follows that
\begin{align*}
\frac{1}{(1+r_j|x_{j,m}-x|)^{N}}\le
\frac{1}{(1+r_j|x_{k,\ell}-x_{j,m}|)^{N}}\le
\frac{(r_k/r_j)^{N}}{(1+r_k|x_{k,\ell}-x_{j,m}|)^{N}},
\end{align*}
and we have
\begin{align*}
|\langle\eta_{k,\ell},\psi_{j,m}\rangle| &\le
\frac{C(r_k/r_j)^{\frac{n}{2}}}{(1+r_k|x_{k,\ell}-x_{j,m}|)^{N}}
\int_{\Rn}\frac{r_k^{n}}{(1+r_k|x_{k,\ell}-x|)^{N}}\d x \\
&\le C
\bigg(\frac{r_k}{r_j}\bigg)^{\frac{n}{2}}(1+r_k|x_{k,\ell}-x_{j,m}|)^{-N}.
\end{align*}
\end{proof}
The following estimate in direct space was used to prove Proposition \ref{johnlennon}.
\begin{lemma}\label{le:max}%
Suppose that $0 < r \le 1$ and $N>n/r$. Then for any sequence
$\{s_{k,\ell}\}_{k,\ell\in\Zn}\subset\bC$, and for $x\in Q(j,m)$, we have
\begin{align}
\sum_{\ell\in
\Zn}\frac{|s_{k,\ell}|}{(1+\min(r_k,r_j)|x_{k,\ell}-x_{j,m}|)^N} \le &C
\max\left(\frac{r_k}{r_j},1\right)^{\frac{n}{r}}\notag \\
&\phantom{C}\times {\mathcal M}_r\Big(\sum_{\ell\in
\Zn}|s_{k,\ell}|\,\mathbf{1}_{Q(k,\ell)}\Big)(x),\label{setty}
\end{align}
with $r_k$, $\xi_k$, and $x_{n,\ell}$ defined in Eq.\ \eqref{eq:rk}, and $Q(j,m)$ defined in Eq.\ \eqref{eq:qk}.
\end{lemma}
\begin{proof}
We begin by
considering the case $r_k\le r_j$. We define the sets,
\begin{align*}
A_0&=\{\ell\in \Zn :r_k|x_{k,\ell}-x_{j,m}|\le 1\},\\
A_i&=\{\ell\in \Zn :2^{i-1}<r_k|x_{k,\ell}-x_{j,m}|\le 2^i\},\,\,\, i\ge 1.
\end{align*}
Choose $x\in Q(j,m)$. There exists $C_1>0$ such that $\cup_{\ell\in A_i}Q(k,n) \subset {B}(x,C_12^{i}r_k^{-1})$. Putting $\omega_n:=|B(0,1)|$, we deduce from Eq.\ \eqref{eq:qk} that  $\int \mathbf{1}_{Q(k,n)}(x)\,dx=\omega_n r_k^{-n}$. Hence, we have
\begin{align*}
\sum_{\ell\in A_i}\frac{|s_{k,\ell}|}{(1+r_k|x_{k,\ell}-x_{j,m}|)^N} & \le C 2^{-iN}\sum_{\ell\in A_i}|s_{\ell,\ell}|\le C 2^{-iN}\Big(\sum_{\ell\in A_i}|s_{k,\ell}|^r\Big)^{\frac{1}{r}} \\
 & \le C 2^{-iN}\bigg(\frac{r_k^{n}}{\omega_n}\int_{{B}(x,C_12^{i}r_k^{-1})}\sum_{\ell\in A_i}|s_{k,\ell}|^r\mathbf{1}_{Q(k,\ell)}\bigg)^{\frac{1}{r}}.
\end{align*}
Using by the definition of the maximal operator \eqref{Max1}, we obtain
\begin{align*}
\sum_{\ell\in A_i}\frac{|s_{k,\ell}|}{(1+r_k|x_{k,\ell}-x_{j,m}|)^N} &\le C
2^{i(\frac{n}{r}-N)}\bigg(\frac{r_k^{n}}{2^{in}\omega_d^B}\int_{B(x,C_12^{i}r_k^{-1})}\sum_{\ell\in
A_i}|s_{k,\ell}|^r\mathbf{1}_{Q(k,\ell)}\bigg)^{\frac{1}{r}}\\
&\le C 2^{i(\frac{n}{r}-N)}{\mathcal M}_r\Big(\sum_{\ell\in
\Zn}|s_{k,\ell}|\,\mathbf{1}_{Q(k,\ell)}\Big)(x)
\end{align*}
by using $\sum_{\ell\in \Zn}\mathbf{1}_{Q(k,\ell)}\le L$. Summing over $i\ge
0$ and using $N>n/r$ gives \eqref{setty}. For the second case, $r_k
> r_j$, we redefine the sets,
\begin{align*}
A_0&=\{\ell\in \Zn :r_j|x_{k,\ell}-x_{j,m}|\le 1\}\\
A_i&=\{\ell\in \Zn :2^{i-1}<r_j|x_{k,\ell}-x_{j,m}|\le 2^i\}, i\ge 1.
\end{align*}
As before we have
\begin{align*}
\sum_{\ell\in A_i}\frac{|s_{k,\ell}|}{(1+r_j|x_{k,\ell}-x_{j,m}|)^N}\le& C
2^{-iN}\bigg(\frac{r_{k}^{n}}{\omega_n}\int_{{B}(x,C_12^{i}r_j^{-1})}\sum_{\ell\in
A_i}|s_{k,\ell}|^r\mathbf{1}_{Q(k,\ell)}\bigg)^{\frac{1}{r}} \\
\le& C
2^{i(\frac{n}{r}-N)}\left(\frac{r_k}{r_j}\right)^{\frac{n}{r}}{\mathcal M}_r\Big(\sum_{\ell\in
\Zn}|s_{k,\ell}|\,\mathbf{1}_{Q(k,\ell)}\Big)(x).
\end{align*}
Summing over $i\ge0$ again gives \eqref{setty}.\\
\end{proof}
To prove Proposition \ref{johnlennon} we also used the following estimate in frequency space.
\begin{lemma}\label{le:sum}%
Let $\delta>0$,  $\alpha\in [0,1)$, and   let $r_k$ and $\xi_k$ be defined as in Eq.\ \eqref{eq:rk} for $k\in\Zn$. Then we have.
\begin{enumerate}
    \item[(i)] There exists $C_a>0$ independent of $k\in\Zn$ such that\\
$\displaystyle
\sum_{j\in \Zn} \min\bigg(\bigg(\frac{r_j}{r_k}\bigg)^{n},\bigg(\frac{r_k}{r_j}\bigg)^{\delta}\bigg)
(1+\max(r_k,r_j)^{-1}|\xi_j-\xi_k|)^{-n-\delta}\le C_a.$\\
 \item[(ii)] There exists $C_b>0$ independent of $j\in\Zn$ such that\\
$\displaystyle
\sum_{k\in \Zn} \min\bigg(\bigg(\frac{r_j}{r_k}\bigg)^{n},\bigg(\frac{r_k}{r_j}\bigg)^{\delta}\bigg)
(1+\max(r_k,r_j)^{-1}|\xi_j-\xi_k|)^{-n-\delta}\le C_b.$
\end{enumerate}

\end{lemma}

\begin{proof}
We first consider (i) and begin by partitioning the $j$ indices in the following sets,
\begin{align*}
A_0&=\{j\in \Zn:|\xi_j-\xi_k|\le \rho_1 r_k\} \\
A_i&=\{j\in \Zn:2^{i-1}\rho_1 r_k <|\xi_j-\xi_k|\le 2^i \rho_1 r_k\},\,\,\, i\ge 1,
\end{align*}
with $\rho_1$ defined in Eq.\ \eqref{strawberry}. There exists $C_1>0$ independent of $i$ such that for $j\in A_i$, we have $B(\xi_j,r_j)
\subset B(\xi_k,C_12^i r_k)$, which follows from using \eqref{strawberry} with  $\xi\in {B}(\xi_j,r_j)$:
\begin{align*}
|\xi_k-\xi| \le |\xi_k-\xi_j|+|\xi_j-\xi| &\le 2^i\rho_1 r_k+r_j\\
&\le 2^i\rho_1 r_k+R_1 2^i r_k\\
&= C_1 2^i r_k.
\end{align*}
Next, we divide the sum even further
 by first looking at $r_k\ge r_j$, and by using that the covering $\{{B}(\xi_j,r_j)\}_j$ is admissible, we obtain with $\omega_n:=|B(0,1)|$,
\begin{align*}
\sum_{\substack{j\in A_i \\ j: r_j\le r_k}} \bigg(\frac{r_j}{r_k}\bigg)^{n}& (1+r_k^{-1}|\xi_j-\xi_k|)^{-n-\delta}\\
\le & C2^{-i(n+\delta)} \sum_{\substack{j\in A_i \\ j: r_j\le r_k}} \bigg(\frac{r_j}{r_k}\bigg)^{n}\frac{1}{\omega_n r_j^n}\int_{B(\xi_j,r_j)}\mathbf{1}_{B(\xi_j,r_j)}(\xi)\d \xi\\
\le & C2^{-i(n+\delta)}\frac{1}{\omega_n r_k^n}\int_{{B}(\xi_k,C_12^i r_k)}\sum_{\substack{j\in A_i \\ j: r_j\le r_k}}\mathbf{1}_{{B}(\xi_j,r_j)}(\xi)\d \xi\\
\le & C2^{-i\delta}.
\end{align*}
Summing over $i$ gives the lemma for the $r_k\ge r_j$ part of the sum.
 In a similar way, the result for $r_k<r_j$ follows by using
\begin{align*}
\sum_{\substack{j\in A_i \\ j: r_j> r_k}} \bigg(\frac{r_k}{r_j}\bigg)^{\delta} (1+r_j^{-1}|\xi_j-\xi_k|)^{-n-\delta}
\le\sum_{\substack{j\in A_i \\ j: r_j> r_k}} \bigg(\frac{r_j}{r_k}\bigg)^{n} (1+r_k^{-1}|\xi_j-\xi_k|)^{-n-\delta},
\end{align*}
which is obtained by observing that
$$(1+r_j^{-1}|\xi_j-\xi_k|)^{-n-\delta}\leq \bigg(\frac{r_k}{r_j}+r_j^{-1}|\xi_j-\xi_k|\bigg)^{-n-\delta}=
\bigg(\frac{r_j}{r_k}\bigg)^{n+\delta}(1+r_k^{-1}|\xi_j-\xi_k|)^{-n-\delta}.$$
We now turn to estimate (ii), where we need to make some minor modifications to the general proof. In the case $r_k\geq r_j$, and $\delta\leq n$, we use the estimate
$$\bigg(\frac{r_j}{r_k}\bigg)^n \big(1+r_k^{-1}|\xi_j-\xi_k|\big)^{-n-\delta}
\leq \bigg(\frac{r_k}{r_j}\bigg)^\delta \big(1+r_j^{-1}|\xi_j-\xi_k|\big)^{-n-\delta},
$$
and proceed exactly as above using the modified sets
\begin{align*}
B_0&=\{k\in \Zn:|\xi_j-\xi_k|\le \rho_1 r_j\} \\
B_i&=\{k\in \Zn:2^{i-1}\rho_1 r_j <|\xi_j-\xi_k|\le 2^i \rho_1 r_j\},\,\,\, i\ge 1.
\end{align*}
In case  $r_k\geq r_j$, and $\delta> n$, we note that
$$\big(1+r_k^{-1}|\xi_j-\xi_k|\big)^{-n-\delta}\leq \big(1+r_k^{-1}|\xi_j-\xi_k|\big)^{-2n},$$
to obtain
$$\bigg(\frac{r_j}{r_k}\bigg)^n \big(1+r_k^{-1}|\xi_j-\xi_k|\big)^{-n-\delta}
\leq \bigg(\frac{r_k}{r_j}\bigg)^n \big(1+r_j^{-1}|\xi_j-\xi_k|\big)^{-2n},
$$
and then proceed as in the previous case. We use similar modifications in the case $r_k<r_j$.
\end{proof}

\end{document}